\documentclass[letterpaper, 10 pt, conference]{ieeeconf}  

\IEEEoverridecommandlockouts 

\usepackage{cite}

\usepackage{graphicx}
\usepackage{xcolor}
\usepackage{epstopdf}
\usepackage{fancyhdr}
\usepackage{subcaption}

\usepackage{amsmath,amsthm,amssymb,mathtools,bbm,dsfont,physics,aligned-overset}
\usepackage[T1]{fontenc}
\usepackage{color}
\usepackage[geometry]{ifsym}
\usepackage{siunitx}
\usepackage{multirow}
\usepackage{placeins}
\usepackage{booktabs}
\usepackage{rotating}
\usepackage{array}
\usepackage{dsfont}
\usepackage{bbm}
\usepackage{bbold}
\usepackage{caption,subcaption}
\usepackage[hyphens]{url}

\usepackage[capitalize]{cleveref}
\usepackage{comment}

\usepackage{epsfig}
\usepackage{epstopdf}

\usepackage{xcolor}

\usepackage[bottom]{footmisc}
\usepackage{graphicx}

\usepackage{mathtools}
\usepackage[acronym]{glossaries}

\newtheorem{theorem}{Theorem}
\newtheorem{lemma}[theorem]{Lemma}

\newtheorem{proposition}[theorem]{Proposition}

\theoremstyle{definition}
\newtheorem{definition}{Definition}
\newtheorem{assumption}{Assumption}
\newtheorem{example}{Example}
\newtheorem{remark}{Remark}

\usepackage{stmaryrd}

\usepackage{tikz}
\usepackage{lipsum}

\newcommand{\reals}{\mathbb{R}}

\renewcommand{\d}{\mathrm{d}}

\newcommand{\transpose}[1]{#1^\top}
\newcommand{\T}[1]{\transpose{#1}}

\newcommand{\pseudoinverse}[1]{#1^\dagger}
\newcommand{\pinv}[1]{\pseudoinverse{#1}}

\newcommand{\Q}{\mathbb{Q}}
\renewcommand{\P}{\mathbb{P}}
\newcommand{\empiricalP}[1]{\widehat\P_{#1}}

\newcommand{\diracDelta}[1]{\delta_{#1}}

\newcommand{\setPlans}[2]{\Pi(#1,#2)}
\newcommand{\probSpace}[1]{\mathcal{P}(#1)}

\newcommand{\pushforward}[1]{{#1}_{\#}}
\newcommand{\transportCost}[3]{W^{#1}(#2,#3)}

\newcommand{\ball}[3]{\mathbb{B}_{#1}^{#2}(#3)}

\DeclareMathOperator*{\st}{s.t.}
\newcommand{\DS}{\displaystyle}

\title{\LARGE \bf Wasserstein Tube MPC with Exact Uncertainty Propagation}
\author{Liviu Aolaritei$^*$, Marta Fochesato$^*$, John Lygeros, and Florian Dörfler
\thanks{$^*$: Equal contribution.}
\thanks{The authors are with the Automatic Control Laboratory, Department of Electrical Engineering and Information Technology at ETH Z\"urich, Switzerland {\tt\footnotesize \{aliviu,mfochesato,jlygeros,dorfler\}@ethz.ch}}}

\begin{document}

\maketitle
\thispagestyle{empty}
\pagestyle{empty}

\IEEEpeerreviewmaketitle

\begin{abstract}
We study model predictive control (MPC) problems for stochastic LTI systems, where the noise distribution is unknown, compactly supported, and only observable through a limited number of i.i.d. noise samples. Building upon recent results in the literature, which show that distributional uncertainty can be efficiently captured within a Wasserstein ambiguity set, and that such ambiguity sets propagate exactly through the system dynamics, we start by formulating a novel Wasserstein Tube MPC (WT-MPC) problem, with distributionally robust CVaR constraints. We then show that the WT-MPC problem: (1) is a direct generalization of the (deterministic) Robust Tube MPC (RT-MPC) to the stochastic setting; (2) through a scalar parameter, it interpolates between the data-driven formulation based on sample average approximation and the RT-MPC formulation, allowing us to optimally trade between safety and performance; (3) admits a tractable convex reformulation; and (4) is recursively feasible. We conclude the paper with a numerical comparison of WT-MPC and RT-MPC. 
\end{abstract}

\section{Introduction}
Model Predictive Control (MPC) is a feedback strategy in which a control law is computed online by repeatedly minimizing a predicted performance
index at each time step. For real-world systems, the prediction of the future behaviour is often affected by uncertainty due to inaccurate system models or exogenous noise. This results in a system-model mismatch that can highly deteriorate the system performance and/or safety. In this context, Tube MPC \cite{stochasticMPC} offers a successful framework to control systems affected by uncertainty.

Tube MPC decomposes the system dynamics into two components: (i) a nominal (unperturbed) dynamics, which is utilized for predictions, and (ii) an error dynamics, which lies in a tube (generally a sequence of polytopes around the nominal dynamics) that contains all possible trajectories of the uncertainty. The shape and contractive properties of the tube are generally determined offline, through a fixed pre-stabilizing feedback control law. 

When the noise has bounded support, and in the absence of additional statistical information on the noise (e.g., samples), robust optimization is employed to solve the Robust Tube MPC (RT-MPC) problem \cite{mayne2005robust,langson2004robust}. As the construction results from a worst-case analysis, RT-MPC may often be too conservative. Some methods have emerged in the literature to reduce this conservatism by allowing for more degrees of freedom in the tube construction, e.g., homothetic tube~\cite{rakovic2012homothetic}, elastic tube~\cite{rakovic2016elastic}, parametrized tube~\cite{rakovic2012parameterized}, or interpolating tube~\cite{kogel2020fusing}. Such strategies often increase the robust optimization problem size, resulting in a higher computational cost.

If statistical information about the noise is available, Stochastic Tube MPC (ST-MPC) schemes have been proposed to reduce the conservatism of RT-MPC~\cite{mesbah2016stochastic}. These schemes relax the robust constraints into probabilistic chance constraints, allowing some pre-specified constraint violation probability~\cite{farina2016stochastic}. While generally intractable, such probabilistic constraints can be dealt with via specific approximations or (often conservative) bounds \cite{kouvaritakis2010explicit, cannon2009probabilistic, cannon2010stochastic,lorenzen2016constraint,paulson2019mixed} if the noise distribution is known, or via scenario-based optimization if the noise distribution is only observable through samples~\cite{hewing2019scenario}.

While alleviating some of the RT-MPC conservatism, the two ST-MPC approaches suffer from some major limitations. First, by considering a specific noise distribution (e.g., Gaussian), it fails to guarantee robustness against different (plausible) noise distributions. Secondly, scenario-based ST-MPC is computationally expensive, and generally requires a prohibitive amount of samples for real-time implementation. Moreover, scenario-based ST-MPC fails to robustify against \emph{distribution shifts}~\cite{koh2021wilds}, i.e., when the training noise distribution is different from the testing noise distribution.

To address these shortcomings, more general uncertainty descriptions, which can account for \emph{distributional uncertainty}, i.e., uncertainty about probability distributions, have been recently proposed. Examples include moment ambiguity sets~\cite{van2015distributionally}, Wasserstein ambiguity sets~\cite{yang2020wasserstein, mark2021data, fochesato2022data,zhong2023tube,micheli2022data}, and Total Variation ambiguity sets~\cite{dixit2022distributionally}. Such uncertainty descriptions allow for a higher modeling power, often at the price of being conservative, e.g., moment ambiguity sets, or hard to propagate through the system dynamics, e.g., Wasserstein and Total Variation ambiguity sets. In particular, the latter can often result in crude overestimation or catastrophic underestimation of the distributional uncertainty \cite[Section~III-A]{aolaritei2023capture}.

In this paper, we build upon the recent paradigms of \emph{Wasserstein distributionally robust optimization}~\cite{mohajerin2018data,shafieezadeh2023new} and \emph{propagation of Wasserstein ambiguity sets}~\cite{aolaritei2022uncertainty,aolaritei2023capture}, and we formulate a novel Tube MPC, which we coin \emph{Wasserstein Tube MPC (WT-MPC)}. Assuming that the noise distribution is compactly supported and only observable through a \emph{limited} number of i.i.d.\ noise samples, we show that:
\begin{itemize}
    \item the distributional uncertainty of the state trajectory (inherited from the distributional uncertainty of the noise) can be expressed as the superposition of a deterministic controlled nominal trajectory and an \emph{autonomous Wasserstein tube}. The Wasserstein tube is composed of a sequence of Wasserstein ambiguity sets which are \emph{exactly} (in closed-form) propagated through the system dynamics;

    \item the Wasserstein tube is composed of probability distributions supported on the robust tube (from RT-MPC). Moreover, through one scalar parameter, WT-MPC can interpolate between the data-
    driven formulation based on sample average approximation and RT-MPC. In particular, this property allows WT-MPC to optimally trade between safety and performance;

    \item even in the presence of a small number of samples, WT-MPC can ensure a desired robustness level for the closed-loop system, a smaller closed-loop cost (i.e., increased performance) compared to RT-MPC, and good computational complexity; and

    \item WT-MPC is recursively feasible.
\end{itemize}

The aforementioned benefits of WT-MPC are numerically validated in Section~\ref{sec:numerical}.

\subsection{Mathematical Preliminaries and Notation}
\label{subsec:math}

Throughout the paper, $\probSpace{\mathcal W}$ denotes the space of probability distributions supported on the set $\mathcal W \subseteq \reals^d$, $\diracDelta{x}$ denotes the Dirac delta distribution at $x\in\reals^d$, and $x \sim \P$ denotes the fact that $x$ is distributed according to $\P$. Moreover, $[i:N]$, with $i \leq N$, denotes the set $\{i,\ldots,N\}$, and the symbols $\oplus$ and $\ominus$ denote the Minkowski sum and Pontryagin difference of sets, respectively (see \cite[Section~3.1]{kouvaritakis2016model}). Finally, given a matrix $A$, $\pinv{A}$ denotes its Moore-Penrose pseudoinverse.

In this paper, we focus on two classes of transformations of probability distributions: \emph{pushforward} via linear transformations and the \emph{convolution} with a delta distribution.

\begin{definition}
\label{def:pushforwad}
Let $\P\in\probSpace{\mathcal W}$ and $A  \in \reals^{m \times d}$. The pushforward of $\P$ via the linear map $x\mapsto Ax$ is denoted by $\pushforward{A}\P$, and is defined by $(\pushforward{A}\P)(\mathcal B)\coloneqq \P(A^{-1}(\mathcal B))$, for all Borel sets $\mathcal B\subset A \mathcal W$ (where $A \mathcal W$ is the image of the set $\mathcal W$ through the linear map $A$).
\end{definition}

Intuitively, $\pushforward{A}\P$ is the probability distribution that results from moving the probability mass at $x \in \mathcal W$ to $A x \in A \mathcal W$. Equivalently, if $x\sim\P$, then $\pushforward{A}\P$ is the probability distribution of the random variable $y=Ax$.

\begin{example}
\label{ex:pushforwad:empirical}
Let $\empiricalP{}=\frac{1}{n}\sum_{i=1}^n\delta_{\widehat{x}^{(i)}}$ be an empirical distribution supported on the samples $\{\widehat{x}^{(i)}\}_{i=1}^n$. Then, $\pushforward{A}\empiricalP{}=\frac{1}{n}\sum_{i=1}^n\delta_{A \widehat{x}^{(i)}}$ is empirical as well, supported on the propagated samples $\{A \widehat{x}^{(i)}\}_{i=1}^n$.
\end{example}

Moreover, given $x \sim \P$ on $\reals^d$ and $y \in \reals^d$, $x+y$ is distributed according to the convolution $\delta_y\ast\P$ defined below.

\begin{definition}
\label{def:convolution}
Let $\P\in\probSpace{\reals^d}$ and $y \in \reals^d$. Then, the convolution of $\P$ and $\delta_y$ is denoted by $\delta_y \ast \P$, and is defined by $(\delta_y \ast \P)(\mathcal A) = \P(\mathcal A \ominus y)$, for all Borel sets $\mathcal A\subset\reals^d$.
\end{definition}

Finally, we are interested in probabilistic constraints based on the \emph{conditional value-at-risk (CVaR)}, formally defined as follows. Given $f:\reals^d \to \reals$ and a random variable $x$ distributed according to $\Q$, the CVaR of the random variable $f(x)$ at probability level $1-\gamma$ is defined as
\begin{align}
\label{eq:CVaR}
    \text{CVaR}_{1-\gamma}^{\Q}(f(x)) = \inf_{\tau \in \reals}\; \tau + \frac{1}{\gamma} \mathbb E_{\mathbb Q} \left[ \max\{0,f(x)-\tau\} \right].
\end{align}
More intuition about CVaR will be offered in Section~\ref{subsec:DR-CC}.


\section{Wasserstein Tube MPC}
\label{sec:Wasserstein:DR-MPC}

We consider the discrete-time linear time-invariant system
\begin{align}
\label{eq:stoch:dyn:sys}
\begin{split}
    x_{t+1} &= A x_t + B u_t + w_t\\
    u_t &= K x_t + c_t,
\end{split}
\end{align}
where the matrices $A \in \reals^{d \times d}$, $B \in \reals^{d \times m}$ are known, the initial condition $x_0 \in \reals^d$ is known and deterministic, and the stochastic noise sequence $\{w_t\}_{t\in\mathbb N} \subset \reals^d$ is i.i.d.\ according to an \emph{unknown} distribution $\P$. Moreover, we consider a \emph{fixed stabilizing} feedback gain matrix $K$, i.e., $A+BK$ is Schur stable. We make the following assumption on the noise.

\begin{assumption}~
\label{assump:noise}
\begin{itemize}
    \item[(i)] $\P$ has compact polyhedral support
    \begin{align*}
        \mathcal W = \{\xi \in \mathbb R^d:\, F \xi \leq g\}.
    \end{align*}

    \item[(ii)] The origin belongs to $\mathcal W$, i.e., $0 \in \mathcal W$.

    \item[(iii)] We have access to $n_0$ i.i.d.\ samples $\{\widehat{w}^{(i)}\}_{i=1}^{n_0}$ from $\P$.
\end{itemize}
\end{assumption}

Since we only have access to a finite number of samples from the unknown noise distribution $\P$, we are faced with \emph{distributional uncertainty}, i.e., uncertainty about probability distributions. In what follows, we employ the methods developed in \cite{aolaritei2023capture,aolaritei2022uncertainty}, which lay the foundation to capture and propagate distributional uncertainty in dynamical systems.


\subsection{Capture and Propagate Distributional Uncertainty}
\label{subsec:LTI:capture}

We start by defining, for any $t\in \mathbb N$, the vector $\mathbf{w}_{[t-1]} = \begin{bmatrix}w_{t-1}^\top & \ldots & w_0^\top \end{bmatrix}^\top$. Then, the distributional uncertainty of the state $x_t$ is naturally inherited from the distributional uncertainty of the noise trajectory $\mathbf{w}_{[t-1]}$. Therefore, in order to capture the distributional uncertainty of $x_t$, we first need to capture the distributional uncertainty of $\mathbf{w}_{[t-1]}$, and then to propagate it, through the system dynamics~\eqref{eq:stoch:dyn:sys}, to $x_t$.

We start by constructing $n \in \mathbb N$ noise sample \emph{trajectories} 
\begin{align*}
    \widehat{\mathbf{w}}_{[t-1]}^{(i)}:= \begin{bmatrix}(\widehat{w}_{t-1}^{(i)})^\top & \ldots & (\widehat{w}_0^{(i)})^\top \end{bmatrix}^\top,
\end{align*}
for $i\in[1:n]$, using the $n_0$ available noise samples. Since the noise is i.i.d., we can easily construct such sample trajectories by letting each entry $\widehat{w}_j^{(i)}$, with $j \in \{0,\ldots,t-1\}$ and $i\in \{1,\ldots,n\}$, be an arbitrary sample from $\{\widehat{w}^{(i)}\}_{i=1}^{n_0}$. We then define the empirical probability distribution on the product set $\mathcal W^{t} = \mathcal W \otimes \ldots \otimes \mathcal W$, with $t$ terms, 
\begin{align*}
    \empiricalP{[t-1]} := \frac{1}{n} \sum_{i=1}^n \delta_{\widehat{\mathbf{w}}_{[t-1]}^{(i)}}.
\end{align*}

Following the approach presented in \cite{aolaritei2023capture,aolaritei2022uncertainty}, we capture the distributional uncertainty of $\mathbf{w}_{[t-1]}$ via \emph{Wasserstein ambiguity sets}, i.e., balls of probability distributions, defined using the Wasserstein distance, and centered at the empirical distribution $\empiricalP{[t-1]}$. 
For $\Q\in\probSpace{\mathcal W^t}$, the (type-$1$) \emph{Wasserstein distance}~\cite{OT} between $\Q$ and $\empiricalP{[t-1]}$ is defined by  
\begin{align*}
    \transportCost{\|\cdot\|_2}{\Q}{\empiricalP{[t-1]}}
    \coloneqq 
    \inf_{\pi\in\Pi}\int_{\mathcal W^t\times \mathcal W^t}\|x_1-x_2\|_2 \,\d\pi(x_1,x_2),
\end{align*}
where $\Pi:=\setPlans{\Q}{\empiricalP{[t-1]}}$ is the set of all probability distributions over $\mathcal W^t\times\mathcal W^t$ with marginals $\Q$ and $\empiricalP{[t-1]}$. The semantics are as follows: we seek the minimum cost to transport the probability distribution $\Q$ onto the probability distribution $\empiricalP{[t-1]}$, when transporting a unit of mass from $x_1$ to $x_2$ costs $\|x_1-x_2\|_2$. Intuitively, $\transportCost{\|\cdot\|_2}{\Q}{\empiricalP{[t-1]}}$ quantifies the discrepancy between $\Q$ and $\empiricalP{[t-1]}$ and it naturally provides us with a definition of ambiguity in $\probSpace{\mathcal W^t}$. Specifically, the \emph{Wasserstein ambiguity set} (henceforth simply referred to as \emph{ambiguity set}) of radius $\varepsilon$, centered at $\empiricalP{[t-1]}$, and with support $\mathcal W^t$ is defined by 
\begin{align*}
    \ball{\varepsilon}{\|\cdot\|_2}{\empiricalP{[t-1]}}
    \coloneqq 
    \{\Q\in\probSpace{\mathcal W^t}: \transportCost{\|\cdot\|_2}{\Q}{\empiricalP{[t-1]}}\leq\varepsilon\}.
\end{align*}
In words, $\ball{\varepsilon}{c}{\empiricalP{[t-1]}}$ includes all probability distributions on $\mathcal W^t$ onto which $\empiricalP{[t-1]}$ can be transported with a budget of at most $\varepsilon$. Such ambiguity sets are shown in \cite{aolaritei2023capture,aolaritei2022uncertainty} to be a very natural and principled tool to capture distributional uncertainty, enjoying powerful geometrical, statistical, and computational features and guarantees. Moreover, they are easily propagated through linear maps, and the result of the propagation is itself an ambiguity set.

\begin{remark}
The ambiguity radius $\varepsilon$ is a tunable parameter that encapsulates the robustness (or risk aversion) level. Higher $\varepsilon$ translates to more distributions being captured within $\ball{\varepsilon}{\|\cdot\|_2}{\empiricalP{[t-1]}}$, and consequently more robustness against unforeseen noise realizations being introduced. In \cite[Lemma~3]{aolaritei2023capture}, the authors explain how to choose $\varepsilon$ and the noise sample trajectories $\{\widehat{\mathbf{w}}_{[t-1]}^{(i)}\}_{i=1}^n$ to guarantee that the underlying (unknown) true distribution of the noise trajectory $\mathbf{w}_{[t-1]}$ is contained in $\ball{\varepsilon}{\|\cdot\|_2}{\empiricalP{[t-1]}}$ with high probability. 
\end{remark}

We are now ready to study the propagation of the distributional uncertainty from the noise $\mathbf{w}_{[t-1]}$ to the state $x_t$. With the aim of formulating a Tube MPC (see \cite[Chapter~3.2]{kouvaritakis2016model}), we start by rewriting the state as $x_t = z_t + e_t$, i.e., the sum of a deterministic \emph{nominal state} $z_t$, and a stochastic \emph{error state} $e_t$. This gives rise to the equivalent system dynamics
\begin{subequations}
\label{eq:stoch:dyn:sys:x}
\begin{align}
\label{eq:stoch:dyn:sys:z}
    z_{t+1} &= A z_{t} + B v_{t}\\
\label{eq:stoch:dyn:sys:v}
    v_{t} &= K z_{t} + c_{t}\\
\label{eq:stoch:dyn:sys:e}
    e_{t+1} &= A_K e_{t} + w_{t},
\end{align}
\end{subequations}
with $A_K := A+BK$, and with initial conditions $z_{0} = x_0$ and $e_{0} = 0$. Such state separation will allow us to represent the distributional uncertainty of the state trajectory as the superposition of the deterministic controlled nominal state trajectory $z_t$ and an \emph{autonomous Wasserstein tube}. To do so, we start by rewriting the error dynamics~\eqref{eq:stoch:dyn:sys:e} in the form
\begin{align}
\label{eq:stoch:dyn:sys:e:0-t}
\begin{split}
    e_{t} &= \mathbf{D}_{t-1} \mathbf{w}_{[t-1]} \\
    \mathbf{D}_{t-1} &:= \begin{bmatrix} I & A_K & \ldots & A_K^{t-1} \end{bmatrix}.
\end{split}
\end{align}
Moreover, we denote by $\{\widehat{e}_t^{(i)}\}_{i=1}^n$ the $n$ error state samples, obtained by feeding into~\eqref{eq:stoch:dyn:sys:e:0-t} the $n$ noise sample trajectories $\{\widehat{\mathbf{w}}_{[t-1]}^{(i)}\}_{i=1}^n$, i.e., \begin{align*}
    \widehat{e}_t^{(i)} := \mathbf{D}_{t-1} \widehat{\mathbf{w}}_{[t-1]}^{(i)}, \quad \forall i \in [1:n].
\end{align*}

The following proposition shows that the distributional uncertainty of $x_t$ can be exactly captured.

\begin{figure}
\centering
\includegraphics[width=\linewidth]{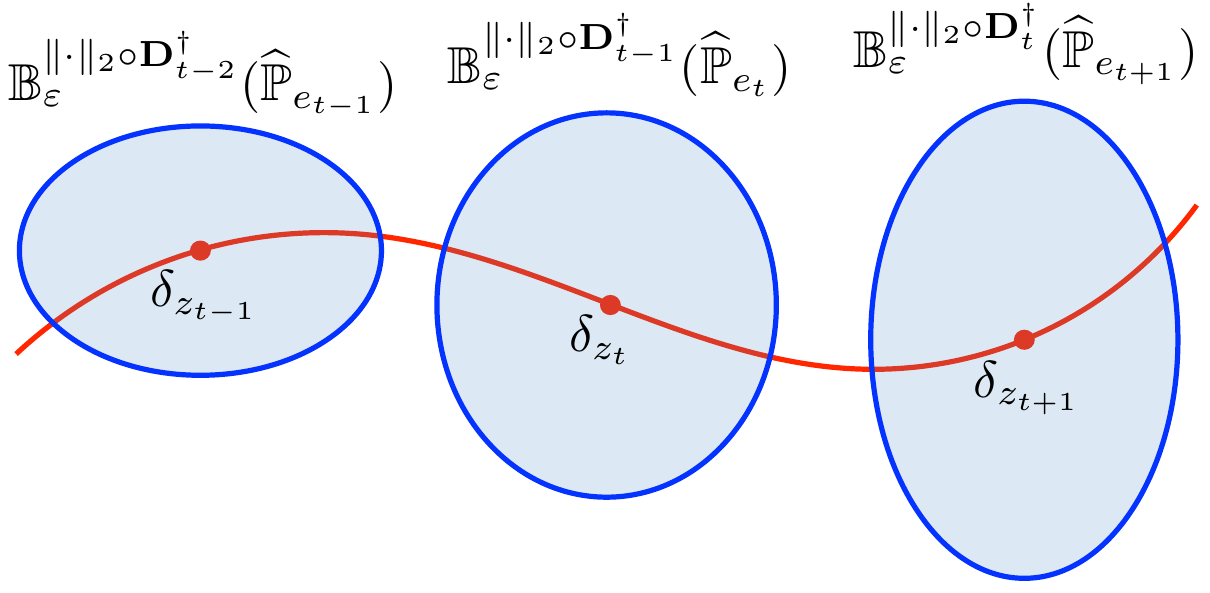}
\caption{Illustration of the nominal state trajectory $\delta_{z_t}$ in $\mathcal P(\reals^d)$, surrounded by the Wasserstein tube $\ball{\varepsilon}{\|\cdot\|_2\circ \pinv{\mathbf{D}_{t-1}}}{\widehat{\P}_{{e}_t}}$, which captures the distributional uncertainty. Together, they capture exactly the distributional uncertainty of the state $x_t$.}
\label{fig:prop:amb:set}
\end{figure}

\begin{proposition}
\label{prop:ambiguity:set:state}
Let Assumption~\ref{assump:noise} hold, and consider the linear control system~\eqref{eq:stoch:dyn:sys}, with i.i.d.\ noise $\{w_t\}_{t\in\mathbb N}$. Moreover, let $\ball{\varepsilon}{\|\cdot\|_2}{\empiricalP{[t-1]}}$ capture the distributional uncertainty of the noise trajectory $\mathbf{w}_{[t-1]}$. Then the distributional uncertainty of $x_t$ is exactly captured by the ambiguity set
\begin{align}
\label{eq:ambiguity:x_t}
    \mathbb{S}_t := \ball{\varepsilon}{\|\cdot\|_2\circ \pinv{\mathbf{D}_{t-1}}}{\widehat{\P}_{{x}_t}},
\end{align}
centered at $\widehat{\P}_{{x}_t}:= \frac{1}{n}\sum_{i=1}^n \delta_{z_t+\widehat{e}_t^{(i)}}$, and with support $z_t \oplus \mathcal E_t$, for $\mathcal E_t := \mathbf{D}_{t-1} \mathcal W^t$. 
\end{proposition}
\begin{proof}
    See Section~\ref{appendix:amb:set:x}.
\end{proof}

\begin{remark}
\label{remark:stoch:tubes}
Expression~\eqref{eq:ambiguity:x_t} reveals that the distributional uncertainty of the state trajectory can be represented in the probability space $\probSpace{\reals^d}$ as the superposition of the nominal state trajectory $\delta_{z_t}$ (the probabilistic representation of $z_t$) and the \emph{Wasserstein tube} $\ball{\varepsilon}{\|\cdot\|_2\circ \pinv{\mathbf{D}_{t-1}}}{\widehat{\P}_{{e}_t}}$, with center $\widehat{\P}_{{e}_t}:=\frac{1}{n}\sum_{i=1}^n \delta_{\widehat{e}_t^{(i)}}$ and support $\mathcal E_t$. This follows immediately from \cite[Theorem~3.7]{aolaritei2022uncertainty}, which ensures that
\begin{align*}
    \mathbb{S}_t = \delta_{z_t} \ast \ball{\varepsilon}{\|\cdot\|_2\circ \pinv{\mathbf{D}_{t-1}}}{\widehat{\P}_{{e}_t}}.
\end{align*}
We illustrate the Wasserstein tube in Figure~\ref{fig:prop:amb:set}. Wasserstein tubes are a natural generalization of the standard robust tubes  (used in RT-MPC), as explained next. Let $\text{diam}(\mathcal E_t)$ denote the diameter of $\mathcal E_t$, measured using the distance $\|\cdot\|_2\circ \pinv{\mathbf{D}_{t-1}}$. Then, through the choice of $\varepsilon$, ranging from $0$ to $\text{diam}(\mathcal E_t)$, the Wasserstein tube $\ball{\varepsilon}{\|\cdot\|_2\circ \pinv{\mathbf{D}_{t-1}}}{\widehat{\P}_{{e}_t}}$ interpolates between the empirical distribution $\frac{1}{n}\sum_{i=1}^n \delta_{\widehat{e}_t^{(i)}}$ and the probabilistic representation of the robust tube $\mathcal E_t$, i.e., the set of all Dirac delta distributions $\delta_\xi$, $\forall \xi \in \mathcal E_t$.  
\end{remark}

In what follows, we inspect the four components of the ambiguity set~\eqref{eq:ambiguity:x_t}.

\begin{itemize}
    \item \textit{Ambiguity radius $\varepsilon$}. This quantity is naturally inherited from the ambiguity set $\ball{\varepsilon}{\|\cdot\|_2}{\empiricalP{[t-1]}}$ that models the distributional uncertainty of the noise trajectory.

    \smallskip  
    \item \textit{Center $\widehat{\P}_{{x}_t}$}. This is an empirical distribution over the $n$ points $\{z_t+\widehat{e}_t^{(i)}\}_{i=1}^n$. Notice that the position of these points in $\mathbb R^d$ is controlled by the feedforward input $c_t$.

    \smallskip  
    \item \textit{Transportation cost $\|\cdot\|_2\circ \pinv{\mathbf{D}_{t-1}}$}. This is defined as
    \begin{align*}
        \left(\|\cdot\|_2\circ \pinv{\mathbf{D}_{t-1}}\right)(\xi) := \|\pinv{\mathbf{D}_{t-1}} \xi\|_2,
    \end{align*}
    and it influences the shape of the ambiguity set, as explained next. Using the SVD decomposition $\mathbf{D}_{t-1} = U \Sigma V^\top$, with $\{\sigma_i\}_{i=1}^d$ the singular values of $\mathbf{D}_{t-1}$ and $\{u_i\}_{i=1}^d$ the orthonormal columns of $U$, the transportation cost boils down to
    \begin{align}
    \label{eq:amb:set:cost}
        \|\pinv{\mathbf{D}_{t-1}}(x_1-x_2)\|_2 = \sqrt{\sum_{i=1}^d \frac{1}{\sigma_i^2} \left|u_i^\top (x_1-x_2)\right|^2}.
    \end{align}
    This shows that the cost of moving probability mass from the center distribution in the direction $u_i$ costs $\norm{x_1-x_2}/\sigma_i$. The feedback gain matrix $K$ has an indirect influence on the amount of mass moved in this direction through the singular value $\sigma_i$ of the matrix $\mathbf{D}_{t-1}$. Specifically, the higher the value of $\sigma_i$, the more probability mass is moved in the direction $u_i$. Similarly, the lower the value of $\sigma_i$, the less probability mass is moved in the direction $u_i$.

    \smallskip  
    \item \textit{Support set $z_t \oplus \mathcal E_t$}. Since $\mathcal W$ is compact and polyhedral, we have that the set $z_t \oplus \mathcal E_t$ is compact and polyhedral, and can be written as
    \begin{align}
    \label{eq:support:amb:set}
        z_t \oplus \mathcal E_t = \{\xi \in \mathbb R^d:\, F_t \xi \leq g_t + F_t z_t\}.
    \end{align}
    for some $q_t \in \mathbb N$, $F_t \in \reals^{q_t \times d}$, and $g_t \in \reals^{q_t}$.
    Moreover, $F_t$ and $g_t$ can be obtained from the following iteration
    \begin{align*}
        \mathcal E_t = A_K \mathcal E_{t-1} \oplus \mathcal W,\quad
        \mathcal E_0 = \{0\}.
    \end{align*}
\end{itemize}

\begin{remark}
    The careful inspection of the ambiguity set~\eqref{eq:ambiguity:x_t} reveals an interesting phenomenon: the feedforward term $c_t$ can control (through $z_t$) the \emph{position} in $\reals^d$ of the center distribution $\widehat{\P}_{{x}_t}$. However, $c_t$ has no influence over the \emph{shape and size} of the ambiguity set (i.e., the transportation cost and radius). In particular, these are exclusively influenced by the feedback gain matrix of $K$ (through $\pinv{\mathbf{D}_{t-1}}$). For more details on the decomposition of the roles of $c_t$ and $K$, we refer to \cite[Section~IV]{aolaritei2023capture}.
\end{remark}


\subsection{Distributionally Robust CVaR Constraints}
\label{subsec:DR-CC}

Armed with the closed-form expression $\mathbb{S}_t$ for the ambiguity set that captures the distributional uncertainty of the state $x_t$, we can now study how to impose constraints. Specifically, we define the polyhedral constraint set
\begin{align*}
    \mathcal X := \left\{x \in \reals^d:\, \max_{j \in [J]} a_j^\top x + b_j \leq 0,\; J \in \mathbb N\right\},
\end{align*}
and we want to guarantee that, for some $\gamma \in (0,1)$, the distributionally robust CVaR (DR-CVaR) constraint
\begin{align}
\label{eq:CVaR:constraint}
    \sup_{\Q \in \mathbb{S}_t} \text{CVaR}_{1-\gamma}^{\Q}\left(\max_{j \in [J]} a_j^\top x_t + b_j\right) \leq 0
\end{align}
is satisfied (see~\eqref{eq:CVaR} for the definition of CVaR). Such constraints are very natural for control tasks in the face of distributional uncertainty, where safety is of interest. We motivate this in what follows.
\begin{itemize}
    \item \emph{Safety considerations.} CVaR constraints are by now standard in risk averse optimization \cite{shapiro2021lectures}. From \cite{nemirovski2007convex} we know that \eqref{eq:CVaR:constraint} implies the following nonconvex distributionally robust chance constraint (DR-CC):
    \begin{align}
    \label{eq:CC:constraint}
        \inf_{\Q \in \mathbb{S}_t} \Q (x_t \in \mathcal X) \geq 1-\gamma.
    \end{align}
    Using Proposition~\ref{prop:ambiguity:set:state}, this guarantees that $x_t \in \mathcal X$, with probability $1-\gamma$, for all the noise trajectory distributions in $\ball{\varepsilon}{\|\cdot\|_2}{\empiricalP{[t-1]}}$. 
    Additionally, \eqref{eq:CVaR:constraint} guarantees that $x_t \in \mathcal X$ in expectation for the most averse noise realizations of probability $\gamma$. This naturally controls the distance of $x_t$ from the set $\mathcal X$ for the remaining probability $\gamma$. Differently, notice that if we only impose \eqref{eq:CC:constraint}, then $x_t$ could be arbitrarily far from $\mathcal X$ with probability $\gamma$.

    \smallskip  
    \item \emph{Computational tractability.} Using results from distributionally robust optimization (DRO) (see \cite{mohajerin2018data} and \cite{shafieezadeh2023new}), in Proposition~\ref{prop:DR:CVaR} we show that the DR-CVaR constraint~\eqref{eq:CVaR:constraint}  can be exactly reformulated as a finite set of deterministic convex constraints.
\end{itemize}

\begin{proposition}
\label{prop:DR:CVaR}
Let Assumption~\ref{assump:noise} hold. Then, the constraint~\eqref{eq:CVaR:constraint} is equivalent to the following set of convex constraints, whose feasible region we denote by $\Gamma_t$: 
\begin{align*}
&\quad\forall i \in [1:n],\; \forall j \in [1:J+1]:\\
&\quad\;\;\begin{cases}
\tau \in \reals, \lambda \in \reals_+, s_i \in \reals, \zeta_{i j} \in \reals^{q_t}_+ &
\\
\lambda \varepsilon n + \sum_{i}^{n} s_{i=1} \leq 0 &
\\
\alpha_j^\top (z_{t} + \widehat{e}_t^{(i)}) + \beta_j(\tau)  + \zeta_{ij}^\top \left(g_t - F_t \widehat{e}_t^{(i)}\right) \leq s_i &
\\
\left\| \pinv{\mathbf{D}_{t-1}} \left( (\pinv{\mathbf{D}_{t-1}})^\top \pinv{\mathbf{D}_{t-1}} \right)^{-1} \left(F_t^\top \zeta_{ij} - \alpha_j \right) \right\|_2 \leq \lambda, &
\end{cases}
\end{align*}
with $\alpha_j := a_j/\gamma$ and $\beta_j(\tau) := (b_j + \gamma \tau - \tau)/\gamma$, for $j \in [1:J]$, as well as $\alpha_{J+1} := 0$ and $\beta_{J+1}(\tau):=\tau$.
\end{proposition}
\begin{proof}
    See Section~\ref{proof:CVaR:reformulation}.
\end{proof}

Proposition~\ref{prop:DR:CVaR} guarantees that the following equivalence
\begin{align}
\label{eq:equivalence:CVaR:Gamma_t}
    \sup_{\Q \in \mathbb{S}_t} \text{CVaR}_{1-\gamma}^{\Q}\left(\max_{j \in [J]} a_j^\top x_t + b_j\right) \leq 0 \iff z_t \in \Gamma_t
\end{align}
holds. Interestingly, \eqref{eq:equivalence:CVaR:Gamma_t} reveals that the DR-CVaR constraint on the distributionally uncertain state $x_t$ can be equivalently reformulated as a set of deterministic constraints on the nominal state $z_t$. Moreover, the number of these constraints grows \emph{linearly} in the number of noise sample trajectories $n$.

In the rest of the paper, with a slight abuse of notation, we will write ${z} \in \Gamma_t$ to denote the fact that there exist $\tau \in \reals, \lambda \in \reals_+, s_i \in \reals, \zeta_{i j} \in \reals^{q_t}_+$, for $i \in [1:n]$ and $j \in [1:J+1]$, which satisfy the constraints in Proposition~\ref{prop:DR:CVaR} for $z_t = {z}$.


\subsection{MPC with DR-CVaR Constraints}
\label{subsec:MPC:safety}

We are now ready to formulate our Wasserstein Tube MPC (WT-MPC). We let $N \in \mathbb N$ denote the MPC horizon of interest, and we use the subscript $k | t$, for $k \in [0:N]$, to denote the (open-loop) predicted dynamics at time $k$ given the (closed-loop) time step $t$. Moreover, as required by engineering applications, we consider the \emph{robust} constraint
\begin{align*}
    u_t \in \mathcal U,
\end{align*}
 on the input, for all $t \in \mathbb N$. Then, the WT-MPC reads
\begin{alignat*}{3}
    \min \; &  \sum_{t=0}^{N-1} \left(\|z_{k | t}\|^2_Q + \|v_{k | t}\|^2_R\right) 
    \\ 
    \st \; & c_{k | t}, v_{k | t} \in \reals^m, z_{k | t} \in \reals^d &&\;\;\;\;\;\;\forall k \in [0:N]  \\
    & z_{k+1 | t} = A z_{k | t} + B v_{k | t} &&\;\;\;\;\;\;\forall k \in [0:N-1] \\
    & v_{k | t} = K z_{k | t} + c_{k | t} &&\;\;\;\;\;\;\forall k \in [0:N-1] \\
    & v_{k | t} \in  \mathcal{U} \ominus K \mathcal E_k &&\;\;\;\;\;\;\forall k \in [0:N-1] \\
    &  z_{k | t} \in \mathcal Z_k &&\;\;\;\;\;\;\forall k \in [1:N-1] \\ 
    & z_{N | t} \in \mathcal Z_f \\ 
    & z_{0 | t} = x_t,
\end{alignat*}
with
\begin{itemize}
    \item \emph{Initial condition $z_{0|t} = x_t$}. The open-loop nominal state $z_{k|t}$ is initialized at the (measured) closed-loop state value $x_t$ (leaving $e_{0 | t} = 0$). This guarantees that the open-loop dynamics are exactly as in \eqref{eq:stoch:dyn:sys:x}.
    \smallskip  
    \item \emph{Nominal constraint sets $\mathcal Z_k$}. The choice $\mathcal Z_k := \Gamma_k$ guarantees that the DR-CVaR constraint~\eqref{eq:CVaR:constraint} is satisfied by $x_{k|t}$ (recall the equivalence~\eqref{eq:equivalence:CVaR:Gamma_t}). In this case, we would like to highlight that the constraint $z_{k | t} \in \Gamma_k$ will be enforced through the set of constraints provided in Proposition~\ref{prop:DR:CVaR}. Consequently, the WT-MPC will inherit the decision variables $\tau \in \reals, \lambda \in \reals_+, s_i \in \reals, \zeta_{i j} \in \reals^{q_t}_+$, $\forall i \in [1:n], \forall j \in [1:J+1]$. However, this choice does not ensure recursive feasibility. In Section~\ref{subsec:recursive:feasibility} we show that this issue can be resolved through an appropriate constraint tightening, which results in a choice $\mathcal Z_k \subsetneq \Gamma_k$. 
    \smallskip  
    
    \item \emph{Terminal nominal set $\mathcal Z_f$}. In Section~\ref{subsec:recursive:feasibility} we explain how this set should be chosen to ensure recursive feasiblity.
     
    \smallskip 
    \item \emph{Input constraint sets $\mathcal{U} \ominus K \mathcal E_k$}. Since
    \begin{align*}
        u_k = K(z_k+e_k)+c_k = v_k + K e_k,
    \end{align*}
    and since $e_k$ is supported on $\mathcal E_k$, this choice guarantees that the input constraint $u_k \in \mathcal U$ is satisfied.
\end{itemize}

Denoting by $(c_{0 | t}^\star,\ldots,c_{N-1|t}^\star)$ the minimizer of the WT-MPC at time step $t$, then the optimal control input applied to system \eqref{eq:stoch:dyn:sys} is $u_t^\star = Kx_{t} + c_{0 | t}^\star$.

\begin{remark}
\label{remark:gamma:relaxation:X-E}
    Notice that the only difference between WT-MPC and RT-MPC stands in the choice of the nominal constraint sets $\mathcal Z_k$. Recall that in RT-MPC the sets $\mathcal Z_k$ are chosen as $\mathcal X \ominus \mathcal E_k$ \cite{mayne2005robust}. Since the DR-CVaR constraint~\eqref{eq:CVaR:constraint} relaxes the set $\mathcal X$, and since \eqref{eq:CVaR:constraint} is equivalent to $z_t \in \Gamma_t$, we have that $\mathcal X \ominus \mathcal E_k \subset \Gamma_k$. Therefore, if $\Gamma_k$ are chosen as nominal constraint sets, we have that WT-MPC reduces the conservatism of RT-MPC (at the expense of a pre-defined constraint violation probability).
\end{remark}

We would like to conclude this section by highlighting the fact that the proposed WT-MPC is a direct generalization of the RT-MPC to the stochastic setting. Indeed, through the choice of $\varepsilon$, the WT-MPC interpolates between the data-driven formulation based on sample average approximation, for $\varepsilon \to 0$, and the RT-MPC formulation, for $\varepsilon \to \text{diam}(\mathcal E_N)$. This follows immediately from Remark~\ref{remark:stoch:tubes}.



\label{sec:closed-loop:properties}



\section{Recursive Feasibility}
\label{subsec:recursive:feasibility}

Since the WT-MPC problem is solved in a receding horizon fashion, it is fundamental to ensure that if it is initially feasible, then it remains feasible at all future time steps. This property is commonly known as \textit{recursive feasibility}. 

As pointed out in the previous section, the choice $\mathcal Z_k = \Gamma_k$, with $\Gamma_k$ defined in Proposition~\ref{prop:DR:CVaR}, does not guarantee recursive feasibility. To see this, recall from \eqref{eq:equivalence:CVaR:Gamma_t} that $z_{k|t} \in \Gamma_k$ is equivalent to $x_{k|t}$ satisfying the DR-CVaR constraint \eqref{eq:CVaR:constraint}. The latter, in turn, guarantees that $x_{k|t} \in \mathcal X$ only with high probability, i.e., for almost (\emph{but not}) all possible noise realizations. Now, since the WT-MPC problem is initialized at the closed-loop state $x_t$, there is a low, but \emph{strictly positive} probability of infeasibility. Notice that this issue does not appear in RT-MPC, which robustifies against all possible noise realizations by imposing the constraint $z_{k|t} \in  \mathcal X \ominus \mathcal E_k$. Such conservative choice, which considers the support of the noise, automatically guarantees recursive feasibility.

In what follows, we will show that an appropriate constraint tightening resolves this issue. We start by defining the nominal constraint sets as
\begin{align}
\label{eq:Z_t}
    \mathcal Z_k := \bigcap_{p=1}^k \left( \Gamma_p \ominus \left(\bigoplus_{r=p}^{k-1} A_K^r \mathcal W \right)\right), \quad \forall k \in [1:N].
\end{align}
By construction, we have that $\mathcal Z_k \subset \Gamma_k$. Such constraint tightening is consistent with the stochastic MPC literature. Specifically, \eqref{eq:Z_t} is closely related to what is done in \cite{kouvaritakis2010explicit} to guarantee recursive feasibility.

\begin{remark}
Using the fact that $\mathcal X \ominus \mathcal E_k$, it can be easily derived that $\mathcal X \ominus \mathcal E_k \subset \mathcal Z_k$. This shows that the constraint tightening \eqref{eq:Z_t}, which guarantees recursive feasibility (see Proposition~\ref{prop:recursive:feasibility}), results in a WT-MPC problem which is generally less conservative than RT-MPC.
\end{remark}

\begin{remark}
The nominal constraints $z_{k |t} \in \mathcal Z_k$ can be easily enforced in the WT-MPC problem. This follows from the characterization of $\Gamma_p$ in Proposition~\ref{prop:DR:CVaR} and from \cite[Theorem~2.2]{kolmanovsky1995maximal}, as explained next. First, we consider the constraint $z_{k |t} \in \Gamma_p$, and notice that it is equivalent to the set of affine constraints
\begin{align}
\label{eq:set:gamma_i:z}
    \alpha_j^\top (z_{k | t} + \widehat{e}_p^{(i)}) + \beta_j(\tau)  + \zeta_{ij}^\top \left(g_p - F_p \widehat{e}_p^{(i)}\right) \leq s_i,
\end{align}
for all $i \in [1:n]$ and $j \in [1:J+1]$, expressed in terms of the auxiliary decision variables $\tau \in \reals, \lambda \in \reals_+, s_i \in \reals, \zeta_{i j} \in \reals^{q_t}_+$. Secondly, since the set $\bigoplus_{r=p}^{k-1} A_K^r \mathcal W$ is polyhedral, we know from \cite[Theorem~2.2]{kolmanovsky1995maximal} that the constraint $z_{k | t} \in \Gamma_p \ominus \left(\bigoplus_{r=p}^{k-1} A_K^r \mathcal W \right)$ is equivalent to the following tightened version of the affine constraints \eqref{eq:set:gamma_i:z},
\begin{align*}
    \alpha_j^\top (z_{k | t} + \widehat{e}_p^{(i)}) + \beta_j(\tau)  + \zeta_{ij}^\top \left(g_p - F_p \widehat{e}_p^{(i)}\right) \leq s_i - h_p(\alpha_j),
\end{align*}
with 
\begin{align*}
    h_p(\alpha_j):= \sup_{\xi \in \bigoplus_{r=p}^{k-1} A_K^r \mathcal W} \alpha_j^\top \xi.
\end{align*}
Notice that the values $h_p(\alpha_j)$ can be computed offline. Finally, the sets $\mathcal Z_k$ in \eqref{eq:Z_t} are defined as the intersection of the sets $\Gamma_p \ominus \left(\bigoplus_{r=p}^{k-1} A_K^r \mathcal W \right)$, and can be easily constructed using the tightened affine constraints introduced above.
\end{remark}

Moreover, for recursive feasibility, we need the following standard assumption on the terminal nominal set.

\begin{assumption}
\label{assump:Z_k:V}
There exists a terminal set $\mathcal Z_f$ satisfying: 
\begin{itemize}
    \item[(i)] $K\mathcal Z_f \subseteq \mathcal{U} \ominus K \mathcal E_{N}$.

    \item[(ii)] $A_K \mathcal Z_f \oplus A_K^N \mathcal W \subseteq \mathcal Z_f$.

    \item[(iii)] $\mathcal Z_f \subseteq \mathcal Z_N$.
\end{itemize}
\end{assumption}

Conditions~(i)-(iii) in Assumption~\ref{assump:Z_k:V} are naturally inherited from robust MPC~\cite{mayne2005robust}, and they guarantee that $\mathcal Z_f$ is forward invariant at time $N$, robustly with respect to the initial condition, for zero control input. Notice that the existence of a set $\mathcal Z_f$ that satisfies condition~(ii) is generally guaranteed, for an appropriately long horizon $N$, by the fact that $A_K$ is Schur stable.

\begin{proposition}
\label{prop:recursive:feasibility}
Let Assumption~\ref{assump:noise} hold, let $\mathcal Z_k$ be defined as in \eqref{eq:Z_t}, and let $\mathcal Z_f$ satisfy Assumption~\ref{assump:Z_k:V}. Then, WT-MPC is recursively feasible. 
\end{proposition}
\begin{proof}
See Section~\ref{proof:recursive:feasibility}.
\end{proof}

In addition to recursive feasibility, whenever the ambiguity set $\ball{\varepsilon}{\|\cdot\|_2}{\empiricalP{[0]}}$ contains the true distribution $\P$ of the noise, Proposition~\ref{prop:recursive:feasibility} guarantees that the closed-loop system satisfies $x_t \in \mathcal X$, with probability $1-\gamma$, $\forall t \in \mathbb N$. This follows immediately from the equivalence~\eqref{eq:equivalence:CVaR:Gamma_t} and the fact that the the DR-CVaR constraint~\eqref{eq:CVaR:constraint} implies the DR-CC~\eqref{eq:CC:constraint}.





\section{Numerical Experiments}
\label{sec:numerical}

\begin{figure}[t]
\centering
\includegraphics[width=\linewidth]{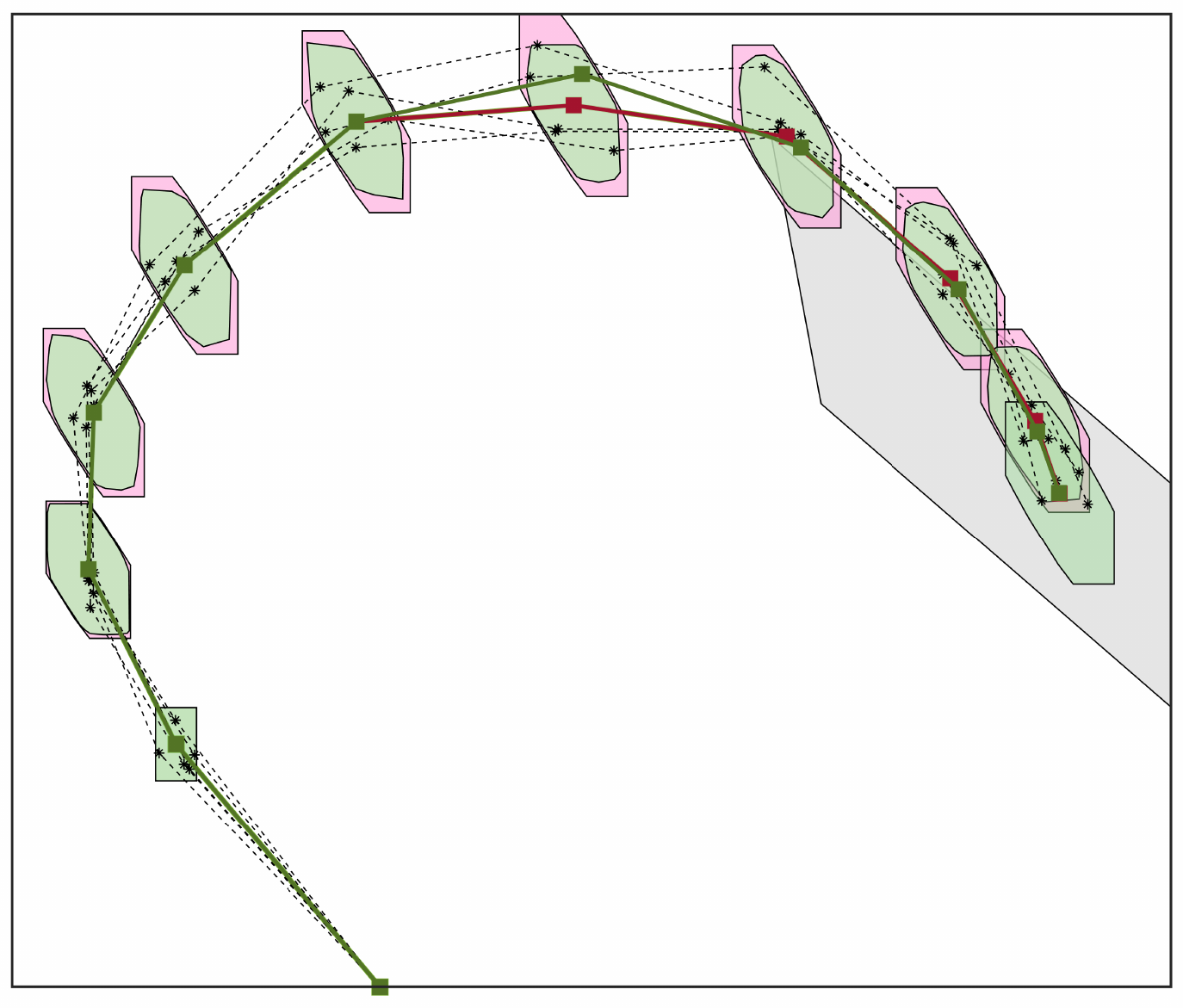}
\caption{Tubes evolution along the prediction horizon. The red solid line denotes the nominal state trajectory under the RT-MPC policy, and the green solid line denotes the one from WT-MPC policy. Black dashed lines represent random realizations of the error trajectory. The pink areas denote the robust tube $\{\mathcal E_k\}_{k=0}^N$ and the light green areas denote the tube $\{\mathcal X - \Gamma_k\}_{k=0}^N$ constructed numerically using Proposition~\ref{prop:DR:CVaR}. For ease of comparison, both tubes are centered around the nominal trajectory from the RT-MPC policy. The grey area represents the terminal invariant set.}
\label{fig:open:loop:behaviour}
\end{figure}

To numerically validate the proposed WT-MPC, we consider the following system
\begin{align*}
    x_{k+1} = \begin{bmatrix} 1 & 1\\ 0 & 1 \end{bmatrix}x_{k} + \begin{bmatrix} 0.5 \\ 1 \end{bmatrix}u_{k} + w_k,
\end{align*} 
borrowed from~\cite{mayne2005robust}, with initial condition $x_0 = [-5,-2]^\top$. We let the support of the uncertainty $\mathcal W$ be the box $[-0.15, 0.15]\otimes[-0.15, 0.15]$, and we consider a noise which is uniformly distributed over $\mathcal W$, i.e., $w_t \sim \mathcal{U}\left(\mathcal W\right)$, $\forall t \in \mathbb N$. Moreover, we assume to have access to a dataset consisting of $n$ disturbance trajectories of length $N=10$, for $n \in \{10,20,50\}$.

For the WT-MPC problem, we consider the stage cost $\|z_{k | t}\|_Q^2 + \|v_{k | t}\|_R^2$ with $Q=\begin{bsmallmatrix} 1 & 0 \\ 0 & 1\end{bsmallmatrix}$, and $R = 0.1$. Moreover, we consider the state constraint set
\begin{align*}
\mathcal X = \{x \in &\reals^2:\, \max \{ \begin{bmatrix}1 & 0 \end{bmatrix} x -2 , \begin{bmatrix}-1 & 0 \end{bmatrix} x - 10, \\   &\begin{bmatrix}0 & 1 \end{bmatrix} x - 2, \begin{bmatrix}0 & -1 \end{bmatrix} x - 2 \} \leq 0 \},
\end{align*}
and the input constraint set $\mathcal{U} = \{u \in \mathbb{R} : -1 \leq u \leq 1\}$. Finally, we set $\gamma = 0.2$ in the DR-CVaR constraint, we choose an ambiguity radius $\varepsilon \in\{0, 0.01, 0.1, 1\}$, and we let $\mathcal Z_k = \Gamma_k$, for all $k \in [1:N]$. We compare:
\begin{itemize}
\item RT-MPC, which robustifies against all possible noise realizations within the support set $\mathcal W$.
\item WT-MPC, which exploits the availability of data (the $n$ noise trajectories), and robustifies against distributional uncertainty by tuning the ambiguity radius $\varepsilon$.
\end{itemize}

\subsection{Open-Loop Analysis}
We first compare the two methods in open-loop. Fig.~\ref{fig:open:loop:behaviour} compares the open-loop behaviour of the system under the RT-MPC policy and the WT-MPC policy. Through the DR-CVaR constraint \eqref{eq:CVaR:constraint}, WT-MPC relaxes the robust constraints into probabilistic ones, allowing a fraction of the state error trajectories $\{e_k\}_{k=0}^N$ to result in a user-defined probability of violating the constraint $x_k \in \mathcal X$. This is equivalent to considering the tube $\{\mathcal X - \Gamma_k\}_{k=0}^N$, which is visibly smaller than to the robust tube $\{\mathcal E_k\}_{k=0}^N$.

In the second open-loop experiment, we compute the empirical probability of violating the constraint $x_t \in \mathcal X$, $\forall t \in [1:N]$, for 10000 noise trajectory realizations. For each parameter configuration $(n,\varepsilon)$ we repeat the procedure 500 times, each time considering a different realization of the center distribution $\empiricalP{[t-1]}$. Fig.~\ref{fig:sensitivity} shows a parametric study of the WT-MPC policy by tuning the ambiguity radius $\varepsilon$ (for fixed $n=20$) and the number of noise sample trajectories $n$ (for fixed $\varepsilon=0.01$). We see that the constraint satisfaction increases with larger radii (i.e., more distributional robustness) and larger sample sizes (i.e., more knowledge about the true noise distribution). However, we observe that $\varepsilon$ has a higher influence on the frequency of violations compared to $n$. This is primarily due to the fact that we use a very limited amount of noise sample trajectories: only $n \in \{10,20,50\}$ from the distribution of $\mathbf{w}_{[N-1]}$, which lives in dimension $Nd = 20$. Using few samples can be highly desirable in real-time applications, since a higher $n$ translates to more constraints in the WT-MPC (see Proposition~\ref{prop:DR:CVaR}), and therefore a higher computational complexity. In that case, Fig.~\ref{fig:sensitivity} shows that this can be done without sacrificing robustness, by simply picking a higher radius. Importantly, increasing $\varepsilon$ comes at no added computational complexity of the WT-MPC problem. 

\begin{figure}
\centering
\includegraphics[width=\linewidth]{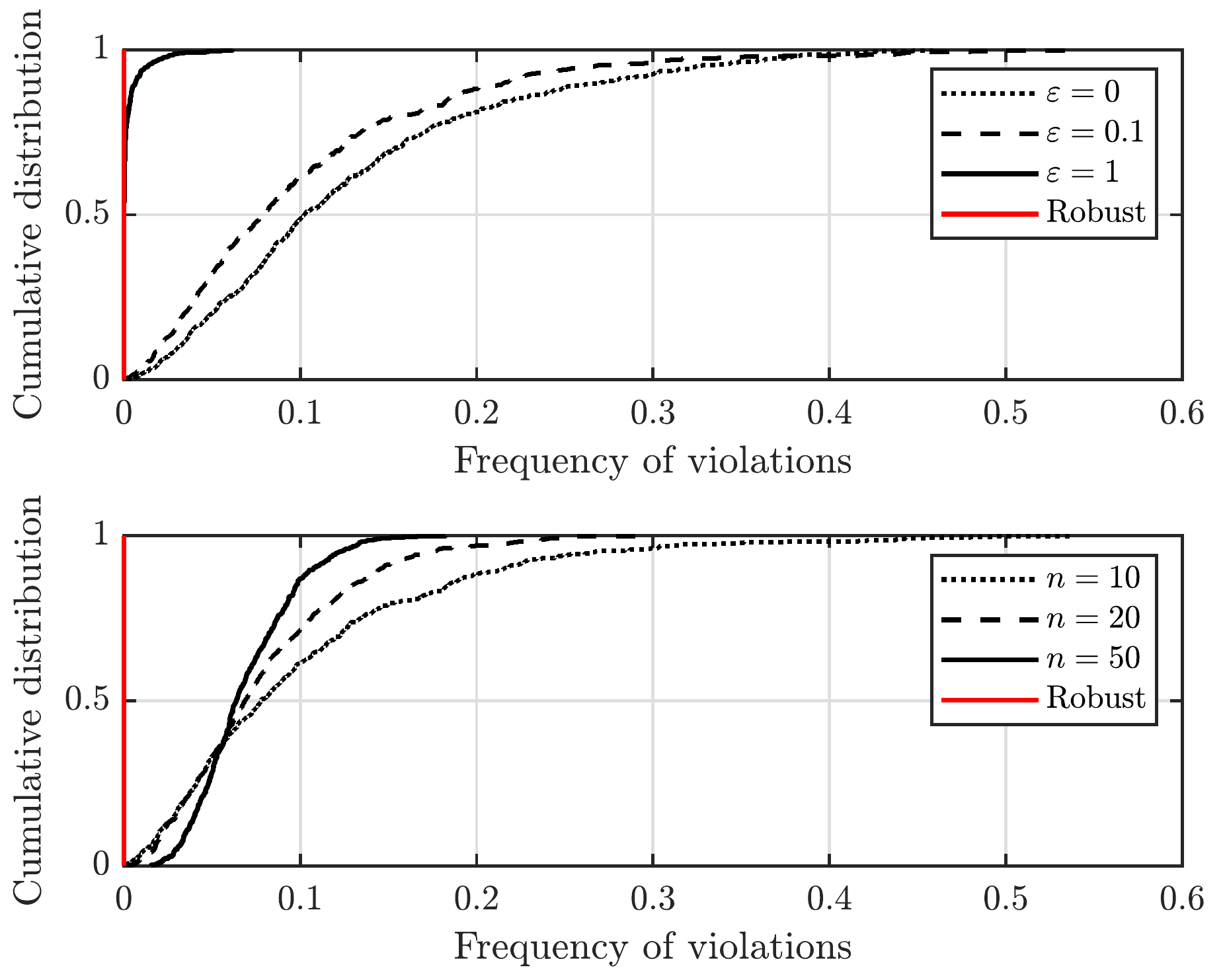}
\caption{Effect of Wasserstein radius (top panel) and number of samples (bottom panel) on the empirical frequency of violation during the open-loop experiments. The case $\varepsilon=0.01$ is almost-identical to $\varepsilon=1$, hence it is not reported to ease the plot readability.}
\label{fig:sensitivity}
\end{figure}

\subsection{Closed-Loop Analysis}

Next, we compare the performance of the two methods in closed-loop for a control task of length $T =15$. Again, we consider $n \in \{10,20,50\}$ (for fixed $\varepsilon=0.01$) and $\varepsilon \in \{0, 0.01, 0.1, 1\}$ (for fixed $n=20$), and for each parameter configuration we repeat the procedure 100 times, each time considering differents realization of the noise. 
The out-of-sample performance in terms of both the empirical probability of violating the constraint $x_t \in \mathcal X$, $\forall t \in [1:T]$, and the closed-loop cost are reported in Fig.~\ref{fig:closedloop}. 

Fig.~\ref{fig:closedloop} strengthens the observation that we have made in the open-loop analysis: even in the presence of a small number of samples $n$, we can ensure 
\begin{itemize}
    \item a desired robustness level for the closed-loop system,

    \item smaller closed-loop cost (i.e., increased performance) compared to RT-MPC, and

    \item good computational complexity,
\end{itemize}
by simply adjusting the value of $\varepsilon$. Moreover, in the bottom plot of Fig.~\ref{fig:closedloop}, we observe that WT-MPC returns a policy which \emph{optimally} trades between safety (i.e., constraint violation) and performance (i.e., closed-loop cost).

\begin{figure}
\centering
\includegraphics[width=\linewidth]{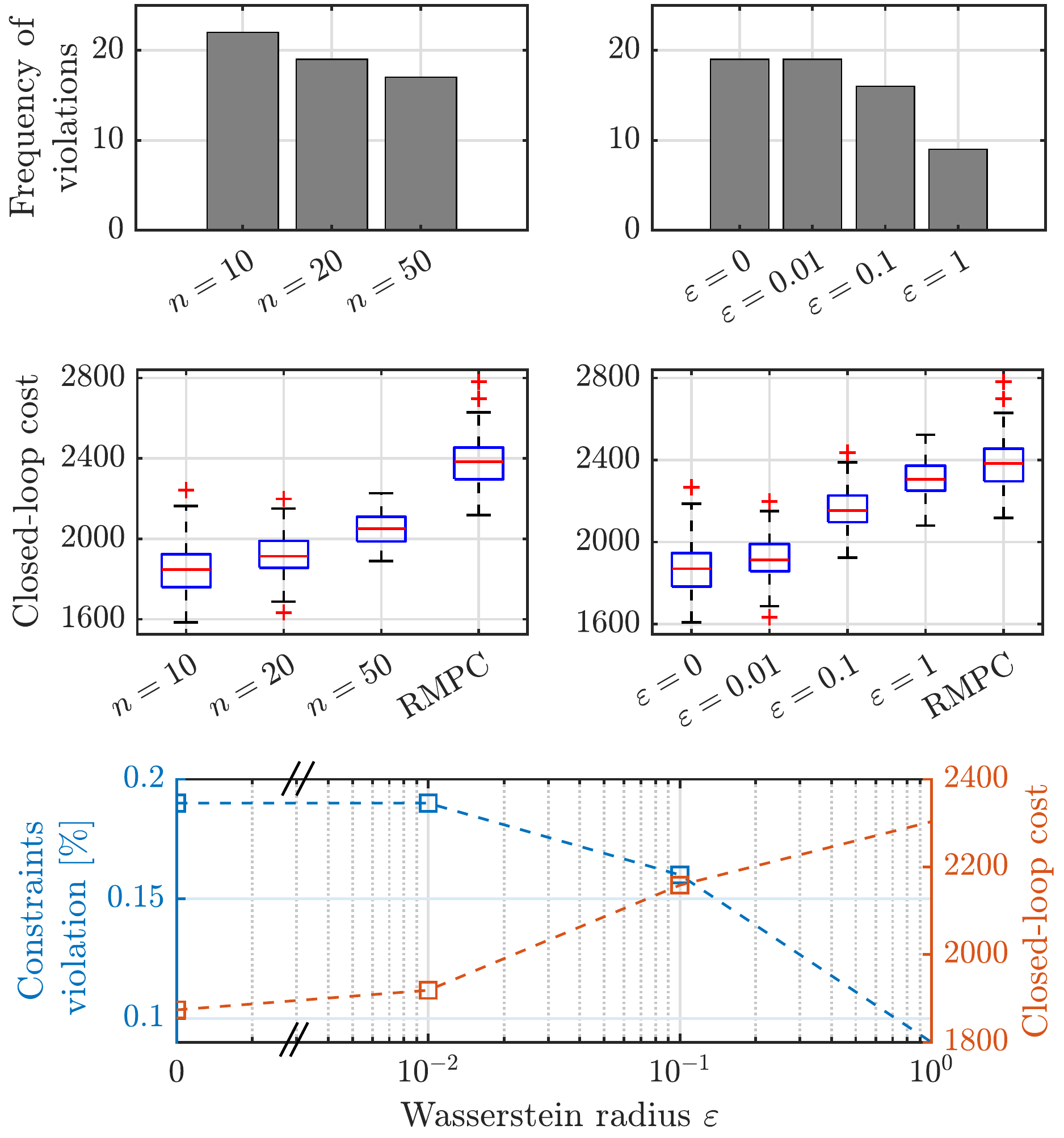}
\caption{Closed-loop analysis. Effect of sample size (top-left panel) and Wasserstein radius (top-right panel) on the frequency of violation during the closed-loop experiments. The robust formulation returns 0 violations in both cases. Closed-loop cost sensitivity to the two tuning knobs is reported in middle-left panel and middle-right panel, respectively. The bottom plot reports the trade-off between safety and performance as a function of Wasserstein radius. }
\label{fig:closedloop}
\end{figure}

\section{Current and Future Work}


Current work focuses on proving the stability properties of the closed-loop system resulting from the WT-MPC policy, and extending the WT-MPC to the case where the stochastic noise has unbounded support. In particular, for the former, we are currently working towards establishing a general framework to prove stability of systems affected by distributional uncertainty.

Future work includes extensions to state feedback MPC (optimizing both over $K$ and $c_t$) and to model uncertainty (see \cite{chen2022robust} for a recent overview of such extensions).

\bibliographystyle{unsrt}
\bibliography{references}

\section*{Appendix}

\subsection{Proof of Proposition~\ref{prop:ambiguity:set:state}}
\label{appendix:amb:set:x}

The proof builds upon \cite[Theorem~2]{aolaritei2023capture} and \cite[Corollary~3.16]{aolaritei2022uncertainty}. We will prove the result in two steps. First, we study the propagation of the distributional uncertainty from the noise $\mathbf{w}_{[t-1]}$ to the error state $e_t$. Secondly, since $z_t$ is deterministic, the distributional uncertainty of $x_t$ follows from a simple convolution with $\delta_{z_t}$ (see Definition~\ref{def:convolution}).
\smallskip

\textbf{Step~1}. Since $\|\cdot\|_2$ is orthomonotone, i.e., $\|x_1+x_2\|_2 \geq \|x_1\|_2$ for all $x_1,x_2 \in \reals^d$ satisfying $\T{x_1}x_2 = 0$, and since $\mathbf{D}_{t-1}$ is full row-rank, \cite[Theorem~2]{aolaritei2023capture}  (see also the subsequent remark) guarantees that
\begin{align*}
    \pushforward{\mathbf{D}_{t-1}}{\ball{\varepsilon}{\|\cdot\|_2}{\empiricalP{[t-1]}}} 
    = 
    \ball{\varepsilon}{\|\cdot\|_2\circ \pinv{\mathbf{D}_{t-1}}}{\pushforward{\mathbf{D}_{t-1}}{\empiricalP{[t-1]}}},
\end{align*}
where the ambiguity set on the right-hand side is supported on $\mathcal E_t = \mathbf{D}_{t-1} \mathcal W^t$. Now, similarly to Example~\ref{ex:pushforwad:empirical}, we have that $\pushforward{\mathbf{D}_{t-1}}{\empiricalP{[t-1]}}$ is precisely the empirical distribution $\widehat{\P}_{{e}_t}:=\frac{1}{n}\sum_{i=1}^n \delta_{\widehat{e}_t^{(i)}}$, with $\widehat{e}_t^{(i)} := \mathbf{D}_{t-1} \widehat{\mathbf{w}}_{[t-1]}^{(i)}$.
\smallskip

\textbf{Step~2}. Since $z_t$ is deterministic, recall from Definition~\ref{def:convolution} that the distributional uncertainty of $x_t$ is exactly captured by the convolution 
\begin{align*}
    \delta_{z_t} \ast \ball{\varepsilon}{\|\cdot\|_2\circ \pinv{\mathbf{D}_{t-1}}}{\widehat{\P}_{{e}_t}}.
\end{align*}
Finally, using \cite[Corollary~3.16]{aolaritei2022uncertainty}, this is equal to the ambiguity set $\mathbb{S}_t$. Indeed, this follows immediately, by noticing that the translation $\xi \to z_t + \xi$ from $\mathcal E_t$ to $z_t \oplus \mathcal E_t$ is a bijective transformation. This concludes the proof.

\subsection{Proof of Proposition~\ref{prop:DR:CVaR}}
\label{proof:CVaR:reformulation}

Using the CVaR definition~\eqref{eq:CVaR}, and defining $\alpha_j,\beta_j$, for $j \in [J+1]$, as in the statement of Proposition~\ref{prop:DR:CVaR}, the left hand-side in the constraint~\eqref{eq:CVaR:constraint} can be rewritten as
\begin{align}
\label{prop:DR:CVaR:1}
    \sup_{\Q \in \mathbb{S}_t} \inf_{\tau \in \reals}\; \mathbb E_{\mathbb Q} \left[ \max_{j \in [J+1]} \alpha_j^\top x_t + \beta_j(\tau) \right].
\end{align}
Since the support set~\eqref{eq:support:amb:set} of $\mathbb{S}_t$ is compact, and since $\|\cdot\|_2\circ \pinv{\mathbf{D}_{t-1}}$ satisfies all the axioms of a norm on $\reals^d$ (which follows from Lemma~\ref{lemma:dual:norm}, using the fact that $\pinv{\mathbf{D}_{t-1}}$ is full-column rank), it can be easily checked that all the assumptions of \cite[Theorem~2.6]{shafieezadeh2023new} hold, and \eqref{prop:DR:CVaR:1} is equivalent to
\begin{align}
\label{prop:DR:CVaR:2}
    \inf_{\tau \in \reals} \sup_{\Q \in \mathbb{S}_t}\; \mathbb E_{\mathbb Q} \left[ \max_{j \in [J+1]} \alpha_j^\top x_t + \beta_j(\tau) \right].
\end{align}

Moreover, from \cite[Corollary 5.1(i)]{mohajerin2018data} or \cite[Proposition~2.12]{shafieezadeh2023new}, we know that the inner supremum in \eqref{prop:DR:CVaR:2} has the following strong dual
\begin{align*}
    \begin{array}{cl}
        \inf & \DS \lambda \varepsilon + \frac{1}{N}\sum_{i=1}^N s_i \\
        \st & \DS \forall i \in [1:N],\; \forall j \in [1:J+1]:\\
        &\begin{cases}
            \DS \lambda \in \reals_+, s_i \in \reals, \zeta_{i j} \in \reals^{q_t}_+ & \\
            \DS \alpha_j^\top (z_{t} + \widehat{e}_t^{(i)}) + \beta_j(\tau)  + \zeta_{ij}^\top \left(g_t - F_t \widehat{e}_t^{(i)}\right) \leq s_i & \\
            \DS \left\| \pinv{\mathbf{D}_{t-1}} \left( (\pinv{\mathbf{D}_{t-1}})^\top \pinv{\mathbf{D}_{t-1}} \right)^{-1} \left(F_t^\top \zeta_{ij} - \alpha_j \right) \right\|_2 \leq \lambda, & 
        \end{cases}
    \end{array}
\end{align*}
where we have used the fact that 
$$\left\| \pinv{\mathbf{D}_{t-1}} \left( (\pinv{\mathbf{D}_{t-1}})^\top \pinv{\mathbf{D}_{t-1}} \right)^{-1} \,\cdot\, \right\|_2$$
is the dual norm of $\| \pinv{\mathbf{D}_{t-1}}\, \cdot\, \|_2 = \|\cdot\|_2\circ \pinv{\mathbf{D}_{t-1}}$ (see Lemma~\ref{lemma:dual:norm}).

Finally, the result in Proposition~\ref{prop:DR:CVaR} follows by noticing that $\inf \cdot \leq 0$ constraints are equivalent to existence constraints.

\subsection{Proof of Proposition~\ref{prop:recursive:feasibility}}
\label{proof:recursive:feasibility}

Let $(c_{0 | t}^\star,\ldots,c_{N-1|t}^\star)$ denote the optimal policy given by the WT-MPC at time step $t$. Moreover, let $(v_{0 | t}^\star,\ldots,v_{N-1|t}^\star)$ and $(z_{1 | t}^\star,\ldots,z_{N|t}^\star)$ denote the corresponding optimal input and nominal states. For recursive feasibility, we will prove that the policy  at time step $t+1$,
\begin{align*}
    (c_{0 | t+1},\ldots,c_{N-1|t+1}) := (c_{1 | t}^\star,\ldots,c_{N-1|t}^\star, 0),
\end{align*}
is feasible. Specifically, we will show that the corresponding $(v_{0 | t+1},\ldots,v_{N-1|t+1})$ and $(z_{1 | t+1},\ldots,z_{N|t+1})$ satisfy the constraints of the WT-MPC problem. 

We start with the nominal states $(z_{1 | t+1},\ldots,z_{N|t+1})$. Since the nominal state initialization at time step $t+1$ is $z_{0 | t+1} = x_{t+1} = z^\star_{1 | t} + w_t$, the following holds for the predicted nominal system,
\begin{align*}
z_{1 | t+1} = A_K(z^\star_{1 | t} + w_t) + B c^\star_{1| t} = z^\star_{2| t} + A_K w_t,
\end{align*}
and, by induction, we obtain
\begin{align}
\label{eq:z_k:z_k+1star}
    z_{k| t+1} = z^\star_{k+1| t} + A_K^k w_t,\quad \forall k \in [1:N-1].
\end{align}
Now, based on the construction of the tightened sets in \eqref{eq:Z_t}, and using \cite[Theorem~2.1]{kolmanovsky1995maximal}, we have that
\begin{align*}
    \mathcal Z_{k+1} \oplus A_K^k \mathcal W \subseteq \mathcal Z_{k},\quad \forall k \in [1:N-1].
\end{align*}
Therefore, the feasibility of the first $N-1$ nominal states follows immediately from the implications
\begin{align*}
    z^\star_{k+1 | t} \in \mathcal Z_{k+1} &\implies z_{k | t+1} \in \mathcal Z_{k},\quad \forall k = 1,\ldots, N-2\\
    z^\star_{N | t} \in \mathcal Z_{f} &\implies z_{N-1 | t+1} \in \mathcal Z_{N-1},
\end{align*}
where the second implication uses the fact that $\mathcal Z_f \subseteq \mathcal Z_N$. 

Finally, the feasibility of $z_{N | t+1}$ follows directly from
\begin{align*}
z_{N | t+1} = A_K z_{N-1 | t+1} &= A_K(z^\star_{N | t} + A_K^{N-1} w_t) \\ &= A_K z^\star_{N | t} + A_K^N w_t
\end{align*}
and the two facts $z^\star_{N | t} \in \mathcal Z_f$ and $A_K \mathcal Z_f \oplus A_K^N \mathcal W \subseteq \mathcal Z_f$, which show that $z_{N | t+1} \in \mathcal Z_f$. 

Consider now the input sequence $(v_{0 | t+1},\ldots,v_{N-1|t+1})$. Again, since $z_{0 | t+1} = z^\star_{1 | t} + w_t$, we have that
\begin{align*}
v_{0 | t+1} &= K(z^\star_{1 | t} + w_t) + c^\star_{1| t} = v^\star_{1 | t} + K w_t,
\end{align*}
and more generally, using \eqref{eq:z_k:z_k+1star}, we have
\begin{align*}
    v_{k| t+1} = v^\star_{k+1| t} + K A_K^k w_t,\quad \forall k \in [0:N-2].
\end{align*}
Now, since $v^\star_{k+1| t} \in \mathcal U - K \mathcal E_{k+1}$, and since
\begin{align*}
    (\mathcal U \ominus K \mathcal E_{k+1}) \oplus K A_K^k \mathcal W \subseteq \mathcal U \ominus K \mathcal E_{k}
\end{align*}
(by \cite[Theorem~2.1]{kolmanovsky1995maximal}), we have that $v_{k| t+1} \in \mathcal U \ominus K \mathcal E_{k}$ for all $k \in [0:N-2]$. This shows that $(v_{0 | t+1},\ldots,v_{N-2|t+1})$ satisfy the input constraints at time $t+1$.

Finally, the feasibility of $v_{N-1 | t+1}$ follows from the assumption $K\mathcal Z_f \subseteq \mathcal{U} \ominus K \mathcal E_{N}$, as explained next. Since
\begin{align*}
    v_{N-1 | t+1} = K z_{N-1|t+1} = Kz^\star_{N | t} + K A_K^{N-1} w_t,
\end{align*}
with $z^\star_{N | t} \in \mathcal Z_f$, and since
\begin{align*}
    K \mathcal Z_f \oplus K A_K^{N-1} \mathcal W &\subseteq (\mathcal{U} \ominus K \mathcal E_{N}) \oplus K A_K^{N-1} \mathcal W \\ &\subseteq \mathcal{U} \ominus K \mathcal E_{N-1},
\end{align*}
we have that $v_{N-1 | t+1} \in \mathcal{U} \ominus K \mathcal E_{N-1}$. This concludes the proof of recursive feasibility.

\subsection{Technical Background Results}

\begin{lemma}
\label{lemma:dual:norm}
Let $A \in \mathbb R^{p \times d}$ be a full column-rank matrix. Then $\|\cdot\|_2 \circ A$ is a norm on $\reals^d$ and $\|\cdot\|_2\circ A(A^\top A)^{-1}$ is its dual norm.
\end{lemma}
\begin{proof}
It is straightforward to check that $\|\cdot\|_2 \circ A$ is a norm on $\reals^d$. In particular, the positive-definiteness property follows from the fact that $A$ is full column-rank. We will now prove that $\|\cdot\|_2\circ A(A^\top A)^{-1}$ is the dual norm to $\|\cdot\|_2 \circ A$, i.e.,
\begin{alignat*}{2}
    \|A(A^\top A)^{-1} y\|_2 = \sup \; &  x^\top y   
    \\ 
    \st \; &  \|A x\|_2 \leq 1.
\end{alignat*}
Using standard Lagrange duality arguments \cite[Proposition~5.3.1]{bertsekas2009convex}, the right-hand side is equal to 
\begin{align*}
    \inf_{\lambda \geq 0} \sup_{x \in \reals^d}  x^\top y +\lambda(1-\|A x\|_2),
\end{align*}
which can be further rewritten as
\begin{align*}
    \inf_{\lambda \geq 0} \sup_{x \in \reals^d} \inf_{\|z\|_2 \leq \lambda}  x^\top y +\lambda-(A x)^\top z.
\end{align*}
Now, using Sion's minimax theorem~\cite[Corollary~3.3]{sion1958general}, we can interchange the supremum and the inner infimum, leading to
\begin{align*}
    \inf_{\substack{\lambda \geq 0 \\ \|z\|_2 \leq \lambda}} \sup_{x \in \reals^d} \lambda + x^\top (y + A^\top z) = \inf_{\substack{\lambda \geq 0 \\ \|z\|_2 \leq \lambda \\ A^\top z = y}} \lambda = \inf_{A^\top z = y} \|z\|_2.
\end{align*}
Finally, since $A^\top$ is full-row rank, the last infimum is attained by the minimum-norm solution $A(A^\top A)^{-1} y$. 
\end{proof}

\end{document}